\documentclass[12]{elsarticle}

\usepackage{amsthm}
\usepackage{amssymb}
\usepackage{amsmath}
\usepackage{comment}
\usepackage{verbatim}

\newtheorem{theorem}{Theorem}[section]
\newtheorem{conjecture}[theorem]{Conjecture}
\newtheorem{lemma}[theorem]{Lemma}
\newtheorem*{definition}{Definition}
\newtheorem{claim}[theorem]{Claim}

\newcommand{\PP}{\mathbb P}
\newcommand{\EE}{\mathbb E}

\begin{document}

%opening
\title{Disproving the normal graph conjecture}

\author[toulouse]{Ararat Harutyunyan\fnref{ANRFQNRT}\corref{cor1}}
\fnref{FQNRT}
\ead{ararat.harutyunyan@lamsade.dauphine.fr}

\author[gscop]{Lucas Pastor\fnref{ANR}}
\ead{lucas.pastor@g-scop.grenoble-inp.fr}

\author[ens]{St\'ephan Thomass\'e\fnref{ANR}}
\ead{stephan.thomasse@ens-lyon.fr}

\cortext[cor1]{Corresponding author}

\fntext[ANR]{These authors were partially supported by ANR Project \textsc{Stint} under Contract \textsc{ANR-13-BS02-0007}.}
\fntext[ANRFQNRT]{This author was partially supported by an FQNRT post-doctoral 
research grant and by ANR Project \textsc{Stint} under Contract \textsc{ANR-13-BS02-0007}.}

\address[toulouse]{LAMSADE, CNRS
Universit\'{e} Paris-Dauphine, PSL Research University\\
75016 Paris, France}
\address[ens]{ENS de Lyon, Laboratoire LIP, Universit\'e de Lyon, France.\\
Institut Universitaire de France.}
\address[gscop]{Laboratoire G-SCOP, University of Grenoble-Alpes, France.}
%\address[iuf]{Institut Universitaire de France.}

%%%%%%%%%%%%%%
%% Abstract %%
%%%%%%%%%%%%%%
\begin{frontmatter}
\begin{abstract}
    A graph $G$ is called normal if there exist two coverings, $\mathbb{C}$ and $\mathbb{S}$ of its vertex set such 
    that every member of $\mathbb{C}$ induces a clique in $G$, every member of $\mathbb{S}$ induces an 
    independent set in $G$ and $C \cap S \neq \emptyset$ for every $C \in \mathbb{C}$ and $S \in \mathbb{S}$. 
    It has been conjectured by De Simone and K\"orner in 1999 that a graph $G$ is normal if 
    $G$  does not contain $C_5$, $C_7$ and $\overline{C_7}$ as an induced subgraph. We disprove this conjecture.
\end{abstract}

\begin{keyword}
    normal graphs \sep perfect graphs \sep random graphs \sep probabilistic method
\end{keyword}
\end{frontmatter}
%%%%%%%%%%%%%%%%%%
%% Introduction %%
%%%%%%%%%%%%%%%%%%
\section{Introduction}

The motivation of the study of normal graphs comes from perfect graphs. A graph $G$ is \emph{perfect} 
if $\chi(H) = \omega(H)$ for every induced subgraph $H$ of $G$. 
Claude Berge first introduced perfect graphs in 1960. His motivation came, in part, from
determining the zero-error capacity of a discrete memoryless channel. This can be formulated as finding the Shannon capacity $C(G)$ of a graph $G$ as follows:
\begin{eqnarray*}
C(G) = \lim_{n \to \infty} \frac{1}{n} \log \omega(G^n)
\end{eqnarray*}
where $G^n$ is the $n^{\text{th}}$ co-normal power of $G$. The \emph{co-normal} product (also called the OR product) $G_1*G_2$ of two graphs $G_1$ and $G_2$
is the graph with vertex set $V(G_1) \times V(G_2)$, where vertices $(v_1, v_2)$ and $(u_1, u_2)$ are adjacent
if $u_1$ is adjacent to $v_1$ or $u_2$ is adjacent to $v_2$. Shannon noticed that
$\omega(G^n) = (\omega(G))^n$ whenever $\omega(G) = \chi(G)$. Since $\omega(G^n) = (\omega(G))^n$ holds for all graphs $G$ with
$\omega(G) = \chi(G)$, one might have expected that perfect graphs are closed under co-normal products. K\"orner and Longo in~\cite{two-step-encoding} proved 
this to be false. This motivated K\"orner \cite{normal-graph} 
to study graphs which are closed under co-normal products.
We note that a \emph{covering} of $G$ is a set of subsets of $V(G)$ whose union
is $V(G)$.

%% Normal graph definition.
\begin{definition}
    A graph $G$ is normal if there exist two coverings, $\mathbb{C}$ and $\mathbb{S}$ of its vertex set such 
    that every member of $\mathbb{C}$ induces a clique in $G$, every member of $\mathbb{S}$ induces an 
    independent set in $G$ and $C \cap S \neq \emptyset$ for every $C \in \mathbb{C}$ and $S \in \mathbb{S}$.
\end{definition}

K\"orner showed that all co-normal products of normal graphs are normal~\cite{normal-graph}.
In the same paper, he also showed that all perfect graphs are normal.
It turns out that normal graphs, like perfect graphs, also have a close relationship with graph entropy. 
The \emph{entropy of a graph} $G$ with respect to a probability distribution $P$ on $V(G)$ is defined as:

\begin{eqnarray*}
    H(G,P) = \lim_{t \to \infty} \min_{U \subseteq V(G^t), P^t(U) > 1 - \epsilon} \frac{1}{t} \log \chi(G^t[U])
\end{eqnarray*}
where $P^t(U) = \sum_{\textbf{x} \in U} \prod_{i=1}^{t}P(x_i)$ and $\epsilon \in (0, 1)$ 
(we note that the limit is independent of $\epsilon$ as shown by K\"orner \cite{K73}). The graph entropy
is sub-additive~\cite{gabor-survey} with respect to complementary graphs:
\begin{eqnarray*}
    H(P) \leq H(G,P) + H(\overline{G},P)
\end{eqnarray*}
for all $G$ and all $P$, where $H(P) = \sum_{i=1}^{n}p_i \log \frac{1}{p_i}$. In fact, the value
\begin{eqnarray*}
    \max_{P}\{H(G,P) + H(\overline{G},P) - H(P)\}
\end{eqnarray*}
is also a measure of how imperfect a graph $G$ is, relating to a parameter introduced
in~\cite{graph-imperfection-0} by McDiarmid called the \emph{imperfection
ratio} of graphs (see also~\cite{graph-imperfection-1} and~\cite{graph-imperfection-2}),
which itself derives its motivation from the radio channel assignment 
problems (see~\cite{graph-imperfection-0} for details). 
% As explained 
% by Gabor in~\cite{gabor-survey}, this new notion can also be characterized in terms of graph entropy. 
%  More precisely, the following holds:
% 
% \begin{eqnarray*}
%     \max_{P}\{H(G,P) + H(\overline{G},P) - H(P)\} = \log imp(G)
% \end{eqnarray*}
% 
% With $imp(G)$ the imperfection ratio of $G$.
In~\cite{normal-entropy} Csisz\'ar et. al showed that:
\begin{eqnarray*}
    H(P) = H(G,P) + H(\overline{G}, P) \text{ for all $P$ if and only if $G$ is perfect}.
\end{eqnarray*}

The relaxed version, i.e., equality holds for at least one $P$, is true whenever $G$ is normal, as shown 
in~\cite{split-entropies}:

\begin{eqnarray*}
    H(P) = H(G,P) + H(\overline{G}, P) \text{ for at least one $P$ if and only if $G$ is normal}
\end{eqnarray*}

% Normal graphs were first introduced by J\'anos K\"orner in~\cite{normal-graph} as 
% it was proved at the same time that all perfect graphs are normal.
% K\"orner and Longo in~\cite{two-step-encoding} gave a complete characterization of weakly splitting graphs.
% 
% %% Weakly splitting theorem.
% \begin{theorem}[\cite{split-entropies} and~\cite{two-step-encoding}]
%     A graph $G$ is weakly splitting if and only if it is normal.
% \end{theorem}

%Normal graphs also have a natural description in terms of entropy that comes from 
%information theory (see~\cite{normal-entropy}). 
It has been proved that 
line-graphs of cubic graphs~\cite{line-graph-cubic}, circulants~\cite{normal-circulant} and a 
few classes of sparse graphs~\cite{normal-sparse} are normal. 
Normal graphs have also been studied for regular and random regular graphs. Hosseini et al. \cite{HMR1, HMR2} have
shown that all subcubic triangle-free graphs are normal as well as that almost all $d$-regular graphs
are normal when $d$ is fixed.

%%% Co-normal product.
%\begin{definition}[Co-normal product]
%    The co-normal product $G = G_1 * G_2$  of graphs $G_1$ and $G_2$ is the graph $G$ such that:
%    \begin{itemize}
%        \item the vertex set of $G$ is the cartesian product $V(G_1) \times V(G_2)$
%        \item two vertices $(u_1, u_2)$ and $(v_1, v_2)$ of $G$ with $u_1, v_1 \in G_1$ and $u_2, v_2 \in G_2$ are adjacent if and only if $u_1v_1 \in E(G_1)$ or $u_2v_2 \in E(G_2)$
%    \end{itemize}
%\end{definition}

By definition it follows that a graph is normal if and only if its complement is normal. 
The simplest graphs that are known to be normal but not perfect are the odd cycles of length 
at least $9$ (see~\cite{normal-graph}). In fact, $C_5$, $C_7$ and $\overline{C_7}$ are 
the only minimally known graphs which are not normal. To this end, De Simone and K\"orner
~\cite{normal-conjecture} conjectured the following.

\begin{conjecture}[The Normal Graph Conjecture]\label{conj:normal-graph}
    A graph with no $C_5$, $C_7$ and $\overline{C_7}$ as induced subgraph is normal.
\end{conjecture}

By analogy with perfect graphs, one can ask whether a graph $G$ is \emph{strongly normal}, i.e.,
every induced subgraph of $G$ is normal. As for perfect graphs, it is natural to try 
to characterize strongly normal graphs by excluding forbidden induced subgraphs. 
This leads to a restatement of Conjecture \ref{conj:normal-graph}.

\begin{conjecture}[\cite{normal-conjecture}]
    A graph $G$ is strongly normal if and only if neither $G$ nor its complement contain 
    a $C_5$ or a $C_7$ as an induced subgraph.
\end{conjecture}

\medskip
In this paper, we disprove the Normal Graph Conjecture. In fact, we prove the following
stronger result.

\begin{theorem} \label{thm:main}
There exists a graph $G$ of girth at least 8 that is not normal.
\end{theorem}

%The next section is devoted to the proof of Theorem \ref{thm:main}
Our proof is probabilistic, i.e., we construct a random graph of girth 8 which is not normal. 
% First, we give some properties of our random graph and then we prove the Theorem~\ref{thm:main}, 
% disproving the Conjecture~\ref{conj:normal-graph}. 
In fact,
our proof method can easily be mimicked to show something stronger: there exist graphs of arbitrary girth $g$
which are not normal. 

The paper is organized as follows. In the next subsection, we introduce the well-known probabilistic
tools that are heavily used in the paper. In Section 2, we state and prove some standard properties of
the random graph $G_{n,p}$ most of which are folklore. In Section 3, using the results of Section 2 and
additional arguments we prove our main result, except a key lemma which is proved in Section 4. 

%%%%%%%%%%%%%%%%%%%%
%% Known theorems %%
%%%%%%%%%%%%%%%%%%%%
\subsection{Probabilistic tools}

To prove our main theorem, we need two basic and well-known probabilistic tools.

\begin{theorem}[Chernoff's Inequality, see \cite{prob-computing}]\label{chernoff}
Let $X_1, \cdots, X_n$ be independent Bernoulli (that is, $0/1$ valued) random variables where $\PP[X_i = 1] = p_i$. Let 
$X = \sum_{i = 1}^{n} X_i$ and let $\mu = \sum_{i = 1}^n p_i$ be the expectation of $X$. Then, for all $0 < \delta < 1$ we have:
\begin{eqnarray*}
    \PP[X \leq (1 - \delta)\mu] \leq e^{-\mu \delta^2 /2} \\
    \PP[X \geq (1 + \delta)\mu] \leq e^{-\mu \delta^2 /3}
\end{eqnarray*}
\end{theorem}

\begin{theorem}[Markov's inequality]
    If $X$ is any non-negative discrete random variable and $a > 0$, then
    \begin{eqnarray*}
        \PP[X \geq a] \leq \frac{\EE[X]}{a}
    \end{eqnarray*}
\end{theorem}

% 
% \begin{theorem}[Union bound]
%     For any events $A_1, A_2, \ldots, A_n$, we have:
%     \begin{eqnarray*}
%         \PP \left[\bigcup_{i=1}^{n} A_i \right] \leq \sum_{i=1}^{n} \PP[A_i]
%     \end{eqnarray*}
% \end{theorem}

%%%%%%%%%%%%%%%%%%%%%%%%%%%%%
%% Random graph properties %%
%%%%%%%%%%%%%%%%%%%%%%%%%%%%%
\section{Random graph properties}

Let $G_{n,p}$ denote the random graph on $n$ vertices in which every edge is randomly and independently 
chosen with probability $p$.

Consider the random graph $G:=G_{n,p}$ with $p = n^{-9/10}$. Denote by
$d := np = n^{1/10}$ and let $X_7$ be the number of cycles in $G$ of length at most $7$. 
By $\alpha(G)$ we denote the size of the
largest independent set in $G$. 
% By $\gamma(G)$ we denote the size of the smallest
% dominating set in $G$, i.e., the size of the smallest set $X \subset V(G)$ such that each vertex
% $v \in V \setminus X$ is adjacent to a vertex in $X$. 
In the sequel, we always
assume that $n$ is sufficiently large. 
%We claim the following. 

\begin{lemma}\label{lem:properties}
    The following properties hold for the graph $G$.

    \begin{itemize}
        \item [(a)] $\PP[X_7 > 4 n^{7/10}] < 1/2$. \\
        \item [(b)] Let $c \geq 10$ be a fixed constant. Then
        $\PP[\alpha(G) \geq cn^{9/10} \log n] 
        \leq n^{- \frac{c^2 n^{0.9} \log n}{3}}$. \\
        %\item [(c)] Let $\epsilon > 0$ be fixed constant. $\PP[\alpha(G) >  \epsilon n] \leq e^{-n}.$ \\
        
%         \item [(d)] Let $0 < c < 1/100$ be a fixed constant. 
%         $\PP[\gamma(G) <  c n^{9/10} \log n] \leq e^{-n^{1-c}/2}$. \\
%         
        \item [(c)] Let $D$ be the event that $G$ has a vertex of degree greater than $2d$.
        Then $\PP[D] \leq e^{-n^{0.1}/10}$.

    \end{itemize}

\end{lemma}

\begin{proof}

    (a) Note that by linearity of expectation,

    $$\EE[X_7] \leq \sum_{l=3}^{7} \binom{n}{l} (l-1)! p^l \leq  \sum_{l=3}^{7} (np)^ l \leq 2 n^{7/10}.$$
    The result now follows by Markov's inequality. 

	\bigskip

    (b) 
    %and (c) 
    is well-known and can be deduced from, for 
    example, Frieze \cite{Frieze90}. We include the proof for completeness.
    By the union bound, we have
     \begin{eqnarray*}
        \PP[\alpha(G(n,p)) \geq x] &\leq& \binom{n}{x} (1-p)^{\binom{x}{2}} \\
                                   &\leq& n^x (e^{-p(x-1)/2})^{x} \leq (ne^{-n^{-0.9}(x-1)/2})^{x}
    \end{eqnarray*}

	Now, setting $x:= cn^{0.9}\log n$ yields the result.
	%or $x:= \epsilon n$ for (c) yields the result.    
    
%    
%    This is a consequence of a well-known result of , 
%    though it can be proved readily using a first moment argument.
%     ***** We actually need to show this but it's ok.

	\bigskip

%     (d) This follows from the idea of the proof of Proposition 2.1 in \cite{GLS}, but as we need a sharper bound we reproduce the proof here for completeness. Denote by $A$ the event that a fixed set $X \subset V$ with $|X| = x$ is a dominating set. Then, by the union bound
%     \begin{eqnarray*}
%         \PP[\gamma(G(n,p)) \leq x] &\leq& \binom{n}{x} \PP[A] \\
%                                    &\leq& n^x (1-(1-p)^x)^{n-x} \leq n^x (1-e^{-px})^{n-x}.
%     \end{eqnarray*}
% 
%     Now, set $x:= cn^{0.9} \log n$. It follows that:
%     \begin{eqnarray*}
%         \PP[\gamma(G(n,p)) \leq x] &\leq& n^{cn^{0.9} \log n} (1 - n^{-c})^{n-cn^{0.9} \log n}
%         \\ &\leq& n^{cn^{0.9} \log n} e^{-0.99n^{1-c}}\\ &<& e^{-n^{1-c}/2}. 
%     \end{eqnarray*}
%     
    (c) Clearly, $\PP[D] \leq n \PP[\deg(v) > 2d]$, where $v$ is some fixed vertex.
    By Chernoff's inequality $\PP[\deg(v) > 2d] \leq e^{-n^{0.1}/3}$. The claim now follows.

\end{proof}
%  
% Now, by Lemma \ref{lem:properties}, the probability that $G$ satisfies none
% of the properties (a), (b), (c) or (d) is at least, say $1/2$, for $n$ sufficiently large.
% Therefore, there exists a graph $G$ satisfying none of the properties (a), (b), (c), (d).
% By (a), we can remove exactly $\lceil 10n^{1/2} \rceil$ vertices from $G$ such that the resulting
% graph $G^{*}$ has girth greater than $g$ (remove one vertex from each of the cycles of length
% at most $g$ and the other vertices arbitrarily). Now, $n^{*} \geq n - 10 n ^{1/2} > 0.9n$, 
% for $n$ sufficiently large. 
% 
% Note that $\gamma(G^{*}) > \gamma(G) - 10n^{1/2}
% > \frac{1}{20g} n^{1- 1/2g} \log n - 10n^{1/2} > \frac{1}{50g} n^{1- 1/2g} \log n $, since $g \geq 3$.
% 
% Also, note that $\alpha(G^{*}) \leq \alpha(G) \leq \frac{3}{2g} n^{1- 1/2g} \log n.$
% 
% Furthermore, property (d) is also preserved for $G^{*}$ since $G^{*}$ is a subgraph of $G$.
% 
% Thus, $G^{*}$ essentially preserves all the properties of $G$. Now, we show that $G^{*}$ is a 
% counterexample to the normal graph conjecture.

%%%%%%%%%%%
%% Proof %%
%%%%%%%%%%%

Let $G$ be a bipartite graph with $m$ edges on vertex bipartition
$(A,B)$. We denote by $d$ its average degree
in $A$, that is $d=m/|A|$ and by $e(X,Y)$ the number of edges
between the set $X$ and $Y$ for any $X \subseteq A$, $Y \subseteq B$. 
A {\it partial cover} of $G$ is a 
set of pairs $(x_i,Y_i)$ where the $x_i$'s are 
distinct vertices of $A$, the $Y_i$'s are disjoint sets 
of $B$, $x_i$ is a neighbor of all vertices of $Y_i$,
the size of each $Y_i$ is $\lceil d/3\rceil$ and finally 
the union of $Y_i$'s has size at least $|B|/3$.

\begin{lemma} \label{lem: bipartite}
Let $G$ be a random bipartite graph on vertex bipartition
$(A,B)$, where each possible edge appears with some probability $p$, independently. If 
$\min\{|A|, |B|\} > 10^{100}p^{-1}$, 
then $G$ has $e(A,B) \in [0.99p|A||B|, 1.01p|A||B|]$ and a partial cover with probability at least $1-e^{-cp|A||B|}$,
where $c > 0$ is an absolute constant.
\end{lemma}    

\begin{proof} 
Let $A'$ be the set of vertices of $A$ with degree in $[0.99p|B|,1.01p|B|]$
in $B$
and $B'$ be the set of vertices of $B$ with degree in $[0.99p|A|,1.01p|A|]$
in $A$.
By Chernoff's inequality, there exists a constant $c > 0$, such that the probability 
that (i) $|A'|<0.99|A|$ or, (ii) $|B'|<0.99|B|$, or (iii) $m:=e(A,B) \notin [0.99p|A||B|, 1.01p|A||B|]$
is at most $e^{-cp|A||B|}$. Indeed, note that
probability of (i) is at most 
$$\binom{|A|}{0.01|A|} (2e^{-(0.01)^2 p|B|/3 })^{0.01|A|} < 2^{|A|} e^{-(0.01)^4p|A||B|} <
e^{-c_1 p|A||B|}$$ for some constant $c_1 > 0$ (here we used the fact that 
$10^{100}p^{-1} < \min\{|A|, |B|\}$). Similarly the probability of (ii)
is at most $2^{|B|} e^{-(0.01)^4p|A||B|} < e^{-c_2 p|A||B|}$ for some constant $c_2 > 0$
(here again we used the fact that $\min\{|A|,|B|\} > 10^{100}p^{-1}$). The probability of (iii)
is clear.

Now, we claim that if $G$ satisfies 
$|A'|\geq 0.99|A|$, $|B'|\geq 0.99|B|$ and  $m \in [0.99p|A||B|, 1.01p|A||B|]$, then it 
has a partial cover. Observe first that at least $3m/4$ edges 
of $G$ must be between $A'$ and $B'$ (call these {\it good edges}). Now greedily pick pairs $(x_i,Y_i)$
where $x_i\in A'$ and $Y_i\subseteq B'\cap N(x_i)$ has size exactly
$\lceil m/3|A|\rceil$ in order to construct a partial cover. If the process stops with 
$Y:=Y_1\cup \dots \cup Y_k$ of size at least $|B|/3$, we have our
partial cover. If not, denote by $X$ the set 
$\{x_1,\dots ,x_k\}$, and note that this implies that every vertex in $A'\setminus X$
has degree less than $\lceil m/(3|A|)\rceil$ in $B'\setminus Y$. Note that 
the size of $X$ is negligible compared to the size of $A'$.
Indeed, $|X|< |B|/\lceil m/(3|A|)\rceil < 4p^{-1} < |A'|/10^{10}$. 
Hence the number of good edges incident to 
$X$ is negligible compared to the number of good edges. In particular,
at least $2.99m/4$ good edges are incident to $A'\setminus X$. However, since 
every vertex in $A'\setminus X$ has degree at most $\lceil m/(3|A|)\rceil$
in $B'\setminus Y$, $e(A' \setminus X, Y) > 2.99m/4-\lceil m/(3|A|)\rceil(|A'|-|X|) > 2.99m/4 - m/3$. 
Now, since $|Y| < |B|/3$, and every vertex in $Y$ has degree at most $1.01p|A|$, it follows that
$e(A' \setminus X, Y) < 1.01 p |A| |B|/3 < 1.01 m / (3 \cdot 0.99)$. This implies that   
$2.99m/4 - m/3 < 1.01 m / (3 \cdot 0.99)$, a contradiction.

% In particular, some vertex $b$
% in $Y$ is incident to at least $3m/(4|Y|)-\lceil m/(3|A|)\rceil(|A'|-|X|))/|Y|$
% % good edges.
% Since $|Y|<|B|/3$, the vertex $b$ has degree at least 
% $9m/(4|B|)-e(|A'|-|X|)/(|B||A|)$. Since $|A'|-|X|\leq |A|$, $b$ has degree at least 
% $9m/(4|B|)-m/|B|$ which is $5m/(4|B|)$. Since the degree of $b$ is at most 
% $1.01p|A|$, this gives $m\leq (4.04/5)p|A||B|$. However, $m \geq 0.99|A| \cdot 0.99p|B|$, 
% since every vertex in $A'$ has degree at least $0.99p|B|$. This yields a contradiction. 
\end{proof}

\section{Proof of Theorem~\ref{thm:main}}

In this section we prove our main result. We say that a graph $G$ admits
a \emph{star covering} if there exist two coverings, $\mathbb{C}$ and $\mathbb{S}$, of $V(G)$ such that: 

\begin{itemize}
    \item [(a)] every member of $\mathbb{C}$ induces a clique $K_2$ or $K_1$ in $G$, where 
  no $K_1$ is included in some $K_2$.
    \item [(b)] the graph on $V(G)$ consisting of the edges of $\mathbb{C}$, denoted by $E[\mathbb{C}]$, is 
    a spanning vertex-disjoint union of stars.
    \item [(c)] every member of $\mathbb{S}$ induces an independent set in $G$.
    \item [(d)] $C \cap S \neq \emptyset$ for every $C \in \mathbb{C}$ and $S \in \mathbb{S}$.
\end{itemize}

Every graph $G$ admitting a star covering is normal, and the converse holds 
for triangle-free graphs: 

\begin{claim} \label{claim: star-covering alpha}
    If $G$ is a normal triangle-free graph, then $G$ admits a star covering $(\mathbb{C}, \mathbb{S})$
    where $E[\mathbb{C}]$ contains at most $\alpha(G)$
    stars.
\end{claim}

\begin{proof}
    Let $(\mathbb{C}', \mathbb{S}')$ be a normal covering
    of $G$. Since $G$ is triangle-free, all cliques in $\mathbb{C}'$ are $K_2$'s
    or $K_1$'s. The cliques $K_1$ included in some $K_2$ can be deleted from $\mathbb{C}'$.
    All that remains to show is that we can reduce to cliques inducing vertex-disjoint stars.
    Indeed, suppose that $E[\mathbb{C}']$ contains two adjacent vertices $u,v$ 
    with $d_{E[\mathbb{C}']}(u) \geq 2$ and $d_{E[\mathbb{C}']}(v) \geq 2$.
    Deleting the edge $uv$ from $\mathbb{C}'$ gives another covering 
    (since $u$ and $v$ are also covered by other edges)
    that is still intersecting with $\mathbb{S}'$.
    Repeating this, we obtain a star covering $(\mathbb{C}, \mathbb{S})$ of $G$.
    
    Now, we show that the number of stars in $E[\mathbb{C}]$ is at most $\alpha(G)$. Indeed, 
    let $x_1,..., x_k$ be the centers of the stars (some centers $x_i$ may be trivial stars) 
    in $E[\mathbb{C}]$, and let $S \in \mathbb{S}$
    be any independent set. 
    Then for each $x_i$, $S$ must contain either $x_i$ or an adjacent 
    neighbor of $x_i$ in $\mathbb{C}$. Since the stars are disjoint, 
    it follows that $k \leq |S| \leq \alpha(G)$.
\end{proof}

Let $G=(V,E)$ be a graph. A {\it star system} $(Q,\cal S)$ of $G$ is a spanning set of vertex 
disjoint stars where $\cal S$ is the set of stars, and $Q$ is the set of centers of 
the stars of $\cal S$. Therefore every $x_i\in Q$ is the center of some star $S_i$ of 
$\cal S$. Moreover, the union of vertices of the $S_i$'s is equal to $V$.
Note that some stars can be trivial, i.e. simply consisting of their center.
To every star system $(Q,\cal S)$, we associate a directed graph $Q^*$ on vertex
set $Q$ by letting $x_i\rightarrow x_j$ whenever a leaf 
of $S_i$ is adjacent to $x_j$. 
Of particular interest here is the following notion
of {\it out-section}: A subset $X$ of $Q$ is an out-section 
if there exists $v$ in $Q$ such that for each $x\in X$, there exists a directed path 
in $Q^*$ from $v$ to $x$. 

Observe that to every star-covering we can associate 
the star-system $E[\mathbb{C}]$.

 \begin{lemma} \label{lem:outsection}
Let $G$ be a normal triangle-free graph with a star covering $(\mathbb{C}, \mathbb{S})$.
We denote by $(Q,\cal S)$ its associated star-system. 
Assume that $X$ is an out-section of $Q^*$. Then the set of leaves of the stars with centers 
in $X$ form an independent set of $G$.
\end{lemma}

\begin{proof}
To see this, consider a vertex $v$ in $Q$ which can reach every vertex $x$ of $X$
in $Q^*$ by an oriented path $v=x_0\rightarrow x_1\rightarrow \dots \rightarrow x_k=x$.
For all $i$, we denote by $S_i$ the star with center $x_i$ (observe that they all have 
leaves, apart possibly
$S_k$). 
Consider an independent set $I$ of $\mathbb{S}$ which contains any leaf of $S_0$. 
Since $I$ is an independent 
set, it does not contain $x_0$, and hence by definition of normal cover
$I$ must contain all the leaves of $S_0$. Now since $x_0\rightarrow x_1$, there 
is a leaf of $S_0$ adjacent to $x_1$. In particular, $x_1$ is not in $I$, implying that 
every leaf of $S_1$ belongs to $I$. Applying the same argument, all leaves 
of $S_i$ belong to $I$, for each $i$. Since this argument can be done for every 
oriented path starting at $v$, any star $S_j$ whose center is reachable from 
$v$ in $Q^*$ by a directed path has all its leaves contained in $I$. In particular, all the leaves with 
centers in $X$ form an independent set.
\end{proof}

This lemma provides a roadmap to a disproof of the normal graph conjecture. Namely,
a normal high girth dense enough random graph will have a star covering  with large 
out-sections, in particular, large independent sets. By tuning the density we can 
contradict the typical stability of such graphs. To achieve this, we need
to introduce the following definitions:

\bigskip

Given a graph $G$ and a subset $Q$ of its vertices partitioned into $Q_1,..., Q_{10}$, we say that
$w \in V\setminus Q$ is a \emph{private neighbor} of a vertex 
$v \in Q_i$ if $w$ is adjacent to $v$ but not to any other vertex in $Q_1,..., Q_i$. 
For each vertex $v$ in some $Q_i$, let $S_v$ be the (possibly trivial) star 
centered at $v$ consisting of $v$ and its private neighbors. Note that by definition, for any distinct vertices
$v, v'$ in $Q$, the
stars $S_v$ and $S_{v'}$ are vertex-disjoint. Thus, $Q$ and the set of 
stars $S_v$ form a star system for the graph induced by the set of vertices in $Q$ and their 
private neighbors. 
We define as previously our oriented graph 
$Q^*$ based on the star system consisting of $Q$ and the set of stars $S_v$. 
Observe that by definition of private neighbors, any arc $u\rightarrow v$ of $Q^*$ with 
$u\in Q_i$ and $v\in Q_j$ satisfies $i<j$. Given $Q_1,...,Q_{10}$ in some graph $G$, 
we refer to this star system as the {\it private} star system over $Q_1,..., Q_{10}$.
The directed graph $Q^*$ is called the {\it private} directed graph over $Q_1,..., Q_{10}$.

Let us now turn to our fundamental property:

\bigskip

\textbf{Property $JQ$:}

\medskip

We say that $G$ satisfies property $JQ$ if for every choice of pairwise disjoint subsets of vertices
$J, Q_1,..., Q_{10}$, with $|J| \leq n^{0.91}$ and $\frac{n^{0.9}}{1000} \leq |Q_i| \leq \frac{n^{0.9}}{500}$
for all  $i=1, \dots ,10$, the private directed graph $Q^*$ over $Q_1,..., Q_{10}$ defined on the induced 
subgraph $G\setminus J$ contains an out-section $X$ such that the sum of the number of private neighbors
corresponding to all the vertices of $X$ is at least $n^{0.95}$.

\bigskip

The crucial point is that a random graph  $G:=G_{n,p}$ with $p = n^{-9/10}$ will 
almost surely have property $JQ$, as claimed by the lemma below.

\begin{lemma} \label{lem:JQ}
$\PP[G  \in JQ] = 1-o(1).$ 
\end{lemma}

We postpone the proof of this lemma to the end of the paper.
Now, we show that Lemmas \ref{lem:properties}, \ref{lem:JQ} and Claim \ref{claim: star-covering alpha}
are sufficient to prove our main theorem.

\begin{proof}[Proof of Theorem \ref{thm:main}]

We consider  a random graph $G:=G_{n,p}$ with $p = n^{-9/10}$.
Using Lemma~\ref{lem:properties} and Lemma \ref{lem:JQ} and the union bound, 
for $n$ sufficiently large, there
exists a $n$-vertex graph $G$ satisfying: (a) $G$ has less than $4n^{0.7}$
cycles of length at most seven, (b) $\alpha (G)<10n^{0.9}\log n$, (c)
$G$ has maximum degree at most $2n^{0.1}$, (d) $G$ has property $JQ$.

Consider a set $S$ of at most $4n^{0.7}$ vertices in $G$ intersecting 
all cycles of length at most 7. Note that $G[V \setminus S]$
has girth at least 8. Assume now for contradiction that 
$G[V\setminus S]$ is a normal graph. By Claim \ref{claim: star-covering alpha},
there is a star 
covering $(\mathbb{C}, \mathbb{S})$ of $G[V\setminus S]$ with the number of stars at most $10n^{0.9}\log n$. 
Let $S'$ be the set of those stars which have size at most $10^{10}\log n$. Let $J = S \cup S'$.
Observe that $|J|\leq 10^{10}\log n \cdot 10n^{0.9}\log n + 
4n^{0.7} < n^{0.91}$. Now, consider $G[V\setminus J]$ and
call $Q$ the set of centers of the remaining stars. 
Observe that the set of stars centered at $Q$ still 
form a star covering of $G[V\setminus J]$. Indeed, $\mathbb{C}$ and $\mathbb{S}$
restricted to $G[V\setminus J]$ is a star covering.

Note that since $|Q| <  10n^{0.9}\log n$, $|V\setminus (J \cup Q)| > n - n^{0.91} - 10 n^{0.9} \log n$. 
Now, since $Q$ is a dominating set in $G[V\setminus J]$, and the degree of every
vertex in $G[V\setminus J]$ is at most $2n^{0.1}$, it follows that $|Q| > \tfrac{n^{0.9}}{3}.$ 

We now define the directed 
graph $Q^*$ on $Q$ based on the star covering of $G[V\setminus J]$. 

\begin{claim}
Every strongly connected component $C$ of $Q^*$ has 
size at most $n^{0.9}/1000$.
\end{claim}

\begin{proof}
Observe that $C$ is an out-section of any of its 
vertices, hence by Lemma \ref{lem:outsection} the set of leaves of stars with centers 
in $C$ is an independent set. Since each star in the star covering of $G[V\setminus J]$ has size at least
$10^{10}\log n$, it follows that $G[V\setminus J]$ has an 
independent set of size $10^{10}\log n \cdot |C|$. The result follows now from 
the fact that $\alpha(G) < 10n^{0.9} \log n$.
\end{proof}

Let $C_1,\dots ,C_k$ be the strongly connected components of $Q^*$,
enumerated in such a way that all arcs $xx'$ of $Q^*$
with $x\in C_i$ and $x'\in C_j$ satisfy $i\leq j$.

We concatenate subsets of the components $C_1, \dots, C_k$ into blocks $Q_1, Q_2, ..., Q_{10}$ with 
$Q_1 = C_1C_2...C_{i_1}$,  $Q_2 = C_{i_1 + 1}...C_{i_2}$,..., $Q_{10} = C_{i_9 + 1}...C_{i_{10}}$
for some $i_1,..., i_{10}$ such that 
for each $Q_i$, $1 \leq i \leq 10$,  $n^{0.9}/1000 \leq |Q_i| \leq n^{0.9}/500$. This is clearly possible
since for each $i \leq k$, $|C_i| < n^{0.9}/1000$ and $|Q| > n^{0.9}/3$. 
   
The crucial remark now is that if a vertex $v$ of $G\setminus (J\cup Q)$
is a private neighbor of a vertex $x_i$ in $Q_i$, then
the edge $x_iv$ must be an edge of the star covering. Indeed,
$v$ has a unique neighbor in $Q_1\cup \dots \cup Q_i$ by definition,
and any edge $vx_j$ where $x_j$ is in $Q\setminus (Q_1\cup \dots \cup Q_i)$
cannot belong to $\mathbb{C}$
since this would imply $x_j\rightarrow x_i$. Now, by property 
$JQ$, we know that the private directed graph $Q'^*$ defined on the stars formed 
by the private neighbors of the $Q_i$'s has an out-section $O$ of size at 
least $n^{0.95}$.
Since $Q'^*$ is a subdigraph of $Q^*$, the set $O$ is also an out-section 
of $Q^*$. Hence the set of leaves with centers in $O$ forms an independent 
set of size $n^{0.95}$ by Lemma \ref{lem:outsection}, contradicting the
fact that $\alpha(G) < 10n^{0.9} \log n$.

% 
% 
% 
% $G-S$ has no star covering for 
% any set $S \subset V$ of order $4n^{0.7}$. By taking $S$ to be a minimal
% set of vertices that intersects every cycle of length at most seven, we have, in particular, that
% there exists a graph $G^{*}$ of girth at least 8, not admitting a star covering. 
% By Claim \ref{claim: star-covering alpha}, this would imply that $G^{*}$ is not normal,
% proving Theorem \ref{thm:main}.  

%Thus, it remains to prove Lemma \ref{lem:JQ}. 

\end{proof}

\section{Proof of Lemma \ref{lem:JQ}}

In this section, we prove Lemma \ref{lem:JQ} to conclude the proof of Theorem \ref{thm:main}.

\begin{proof} 
[Proof of Lemma \ref{lem:JQ}]

We will prove that $\PP[JQ^c] = o(1)$. 
We first fix the sets $J, Q_1,..., Q_{10}$. Note that 
there is at most $\sum_{i=1}^{n^{0.91}} \binom{n}{i}
\leq 2n^{n^{0.91}}$ possible sets for $J$ 
and at most $(\sum_{i=n^{0.9}/1000}^{n^{0.9}/500} \binom{n}{i})^{10} \leq 2^{10}n^{n^{0.9}/50} $
sets for the $Q_1,..., Q_{10}$. Thus, there are at most $2^{11} n^{2n^{0.91}}$ ways to fix
the sets $J, Q_1,..., Q_{10}$. We will recall this fact later; in the sequel, the sets 
$J, Q_1,..., Q_{10}$ are fixed.

% Let $M_1$ be the event that for some \emph{fixed} sets $J, Q_1,...,Q_{10}$
% the property $JQ^c$ holds. Clearly, $\PP[JQ^c] \leq 2^{11} n^{2n^{0.91}} \PP[M_1]$. 
% Now, we bound $\PP[M_1]$.

Denote by $B:= G \setminus \{\cup_{i=1}^{10}Q_i \cup \{J\} \}$. 

Note that $|B| \geq 
n - n^{0.91} - \tfrac{n^{0.9}}{50} \geq n - 2n^{0.91}$.     
For a vertex $v \in Q_1$, let $D_v$ be the number of neighbors of $v$ in $B$.
Let $D_{Q_1}$ be the event that at least $0.01|Q_1|$ vertices $v$ in $Q_1$ have   $D_v \notin (0.95d, 1.01d)$. 
We recall that $n^{0.9}/1000 \leq |Q_1| \leq n^{0.9}/500$.    
    %We claim that $D_{Q_1}$ occurs with probability at most $e^{-n}$. 

    Note that, for some sufficiently small $\delta, \epsilon >0$ 
    \begin{eqnarray*}
        \PP[D_{Q_1}]            &\leq& \binom{|Q_1|}{0.01|Q_1|}(\PP[D_v \notin (0.95d, 1.01d)])^{0.01|Q_1|}\\
        &\leq& \binom{n/500d}{n/50000d}(\PP[D_v \notin (0.95d, 1.01d)])^{n/100000d}\\ 
                                &\leq& (n/500d)^{n/50000d}(e^{-\epsilon d})^{n/10^5d}\\
                                &<& e^{-\delta n}.
    \end{eqnarray*}
    where we used the fact that $D_v$ is a binomial random variable
    with mean $p|B| \in (0.96d, d)$ and thus Chernoff's inequality applies. 
    
    \bigskip
    
    For a vertex $v \in B$, let $X_v$ be the random variable counting the number of vertices in $Q_1$ adjacent to $v$,  
    and $X$ be the number of vertices in $B$ that have degree equal to 1 in $Q_1$. 
    Then $X$ is a binomial random variable. Now,
   \begin{eqnarray*}
       \EE[X] &=& |B| \times \PP[X_v = 1]\\
			 &\geq& 0.96n \PP[X_v = 1]\\
                         &\geq& 0.96n |Q_1| \tfrac{d}{n}(1-d/n)^{|Q_1|-1}\\
                         &\geq& 0.96|Q_1|d e^{-1/250}\\
                         &\geq& 0.95|Q_1|d.
   \end{eqnarray*}
   By Chernoff's inequality, since $\EE[X] \geq 0.95 n /1000$, for some $\delta > 0$ sufficiently small,
   \begin{eqnarray*}
       \PP[\{X < 0.9|Q_1|d\}]  &\leq& e^{-\delta n}.
   \end{eqnarray*}
    
    % \leq \binom{n}{0.01n} (\PP[X_v \neq 1 \mid N_4^*])^{0.01n} $. 
    
%     $|Q_1| \frac{d}{n}(1-d/n)^{|Q_1|} \geq 
%      e^{-1/500}/1000$. Therefore, $\PP[N_4^* \cap X > 0.01 n]
%    \leq \binom{n}{0.01n} (0.999)^{0.01n} \leq n^{0.01n} 10^{-0.03n} \leq e^{-0.01n}$, for $n$ sufficiently large.
%    
    
    \medskip
    
    Next, let $Z_E$ be the number of edges from $Q_1$ to $B$. Note that $Z_E$ is a
        binomial random variable
        with mean $\mu = |Q_1||B|\frac{d}{n}$. Note that $\mu \in (0.96|Q_1|d, |Q_1|d)$.
        Then, for some $\delta > 0$ sufficiently small,
     \begin{eqnarray*}
         \PP[\{Z_E \notin (0.95|Q_1|d, 1.01|Q_1|d)\}] &\leq&  e^{-\delta n},  
     \end{eqnarray*}
     by Chernoff's inequality. Now, let $M$ be the event
     \begin{eqnarray*}
         M:= \{Z_E \in (0.95|Q_1|d, 1.01|Q_1|d)\} \cap D^c_{Q_1} 
         \cap \{X > 0.9|Q_1|d \}.
     \end{eqnarray*}
     
     Clearly, $$\PP[M^c] \leq 3e^{- \delta n}.$$
%      
%      \begin{eqnarray*}
%          \PP[JQ^c] &\leq& 2^{11} n^{2n^{0.91}}\PP[M_1] \\
%                 &\leq& 2^{11} n^{2n^{0.91}} (\PP[M_2] + 3e^{-n/10^7})\\
%                 &\leq& 2^{11} n^{2n^{0.91}} \PP[M_2] + o(1).
%      \end{eqnarray*} 
%      
%      Thus, it suffices to bound $\PP[M_2]$.
       Thus,  $$\PP[M] \geq 1- 3e^{-\delta n}.$$

    \bigskip
%      For $v \in Q_1$, we say that $w \in B$ is a \emph{private neighbor} of 
%     $v$ if $w$ is adjacent to $v$ but not to any other vertex in $Q_1$. 
%     
    Let $N_{Q_1}$ be the event that at least $|Q_1|/2$ vertices in $Q_1$ have at least $d/2$ private neighbors.
    We claim that if the event $M$ holds then so does $N_{Q_1}$.  
    
    Assume that $M$ holds. Let us call an edge $e$ a \emph{good} edge if its endpoint in $Q_1$, 
    say $v$, has $D_v \in (0.95d,1.01d)$ and its endpoint 
    in $B$ has degree exactly 1 in $Q_1$. We compute the number 
    of non-good edges. First, let us count the number of 
    edges whose endpoint in $B$ has degree greater than 1.
    
    Note that the number of vertices in $B$ that have degree 1 in $Q_1$ is at least $0.9 |Q_1|d$.
    These vertices contribute at least $0.9 |Q_1|d$ edges. Thus, the number of edges between
    $Q_1$ and $B$ whose endpoint
    in $B$ is not of degree 1 is at most $1.01|Q_1|d - 0.9|Q_1|d \leq 0.11 |Q_1|d$.
    
    Next, we count the number of edges between $Q_1$ and $B$ 
    whose endpoint in $Q_1$, say $v$, satisfies $D_v \notin (.95d, 1.01d)$.
    Since at least $0.99|Q_1|$ vertices in $Q_1$ have degree in the interval $(.95d, 1.01d)$, 
    they contribute to at least
    $.99 \cdot 0.95 |Q_1|d$ edges. The remaining number of edges is 
    at most $1.01|Q_1|d - 0.99 \cdot 0.95 |Q_1|d \leq 0.07|Q_1|d$. 
    
    Thus, the number of edges which are not good is at most $0.18|Q_1|d$.
%     which implies
%     that the number of good edges is at least $0.98|Q_1|d - 0.08|Q_1|d \geq 0.9|Q_1|d$.  
%      

    Now, we prove our claim that if $M$ holds then $N_{Q_1}$ holds as well. 
    We recall again that at least $0.99|Q_1|$ vertices in $Q_1$ have degree
    at least $0.95d$ in $B$. Let us compute the number of 
    vertices in $Q_1$ (called \emph{bad} vertices) which do not have at least $d/2$ private neighbors.
    By the remark above, the number of bad vertices which have degree at most $0.95d$ in $B$ is at most
    $0.01|Q_1|$. Thus, it suffices to bound the number of bad vertices which have degree at least $0.95d$ in $B$.
    Such a vertex is adjacent to at least $0.45d$ non-good edges since its degree is at least
    $0.95d$. Since the total
    number of non-good edges is at most $0.18|Q_1|d$ it follows that the number of all bad vertices is easily at most 
    $\tfrac{0.18|Q_1|}{0.49} + 0.01|Q_1| < |Q_1|/2$. Therefore, at 
    least $|Q_1|/2$ vertices in $Q_1$ have at least $d/2$ private neighbors, proving
    the claim. Summarizing,
    \begin{eqnarray*}
        \PP[M] &=& \PP[N_{Q_1} \cap M] + \PP[N^c_{Q_1} \cap M]\\
                 &=& \PP[N_{Q_1} \cap M].
    \end{eqnarray*} 
    %So it is sufficient to bound  $\PP[N_{Q_1} \cap M_2]$. 
    Thus, $$\PP[N_{Q_1}] \geq 1- 3e^{-\delta n}.$$
  Now, define $B_2 = B \setminus \Gamma(Q_1)$, where $\Gamma(Q_1)$ is
  the set of neighbors of $Q_1$ in $B$. Define $N_{Q_2}$ to be 
  the event that at least $|Q_2|/2$ vertices in $Q_2$ have at least
  $d/2$ private neighbors in $B_2$. We would like to show that
  $\PP[N_{Q_2}]$ holds with high probability.
  First, note that $\PP[|\Gamma(Q_1)| > n/400] \leq \PP[Z_E > n/400] < e^{-\delta n}$.
  
  Thus, it suffices to bound $\PP[N_{Q_2} \mid \{|B_2| > |B| - n/400\}]$.
  By an identical argument as for $N_{Q_1}$, we know that the probability of this event
  is at least $1 - O(e^{- \delta_2n})$, for some $\delta_2 >0$. Indeed, the only assumption that we need
 that was used before is that $|B_2| \geq 0.96n$, which holds as $|B| \geq n - 2n^{0.91}$. 
 
 Thus, $\PP[N_{Q_2}] \geq (1-e^{-\delta n})(1 - O(e^{- \delta_2n})) \geq 1- e^{-\beta n},$ for some $\beta >0$.
%   
%   as $|B_2| \geq 0.99n$.
% %   is at least 
% %   is at most $O(e^{-n/10^{10}})$ since $Z_E$ is at most of size $n/400$ with this probability
% %   and $|B| \geq n - 2n^{0.91}$. 
%   Thus, it is sufficient to bound $\PP[N_{Q_1} \cap M \cap \{|B_2| > 0.99n\}]$.
%   
  
%   By an identical argument as before, we know the probability of the
%   complement of this event is $O(e^{-n/10^7})$ as $|B_2| \geq 0.99n$. 
%   
  For each $i$,
  $2 \leq i \leq 10$, we define the
  sets $B_i$ by $B_{i+1} :=   B_i \setminus \Gamma(Q_i)$ and $N_{Q_i}$ 
  as the event that at least $|Q_i|/2$ vertices in $Q_i$ have at least
  $d/2$ private neighbors in $B_i$. By repeating the same argument as before
  we obtain that with probability at least $1 - O(e^{-\epsilon n})$ the 
  event $N_{Q_i}$ holds, for some $\epsilon > 0$. 
  Indeed, the size of the $B_i$'s almost surely never decreases by more than
  $n/400$ at a time and thus for each $i$, $|B_i| > |B| - n/40 > 0.97n$, allowing us to guarantee
  that the event $M$ holds with high probability in each iteration.
  
  It follows that $$\PP[(\cap_{i=1}^{10} N_{Q_i})^c] = O(e^{- \epsilon' n}),$$
  for some $\epsilon' >0$.

\medskip

Armed with the fact that the event $\cap_{i=1}^{10} N_{Q_i}$ holds with very
high probability, we will finish the proof. We will say that a vertex $v$ in some $Q_i$ 
is \emph{rich} if $v$ has at least $d/2$ private neighbors; similarly, a set $S$ of vertices
is called \emph{rich} if every vertex of $S$ is rich. 
Let $\tau$ be an ordering of the vertices of $G \backslash J$; we will use
this ordering a bit later.

% From now on, we assume that the event $\cap_{i=1}^{10} N_{Q_i}$ holds.
% This implies that at least $0.5|Q_i| \geq \tfrac{0.5n}{1000d} = \tfrac{n}{2000d}$ 
% vertices of each $Q_i$ have at least $d/2$ private neighbors. 
% By taking an appropriate subgraph if necessary, we may assume that
% each $Q_i$ has size $|Q_i| = \left \lfloor \tfrac{n}{2000d} \right \rfloor$ and each vertex
% in each $Q_i$ has exactly $\left \lfloor \tfrac{d}{2} \right \rfloor$ private neighbors. 
% Let $Q^{*}$ be the private directed graph over $Q_1,..., Q_{10}$.
% We will show that there is an out-section of $Q^{*}$ consisting of some vertices in $Q_{10}$
% whose corresponding private neighbors have combined size of at least $n^{0.95}$.

Consider the following set of events. 
We remark that for our purposes we are only interested in the case $i=10$. 

\bigskip
{\it
There exist positive constants $\epsilon_i$ and $C_i$ such that in each $Q_i$, $2 \leq i \leq 10$,
there exist at least
$\frac{\epsilon_i n}{d^i}$ rich, disjoint out-sections of $Q^{*}$,
each of size at least $\frac{d^{i-1}}{C_i}$.  }

\bigskip
Let $J_i$ be the $i^{th}$ event in the above statement.

We inductively prove the following claim (*): for appropriate
values of $\epsilon_i$ and $C_i$, there exist $\epsilon'_i > 0$ such that
$\PP[J_i] \geq 1 - e^{-\epsilon'_i n}$.

\bigskip
%We have already shown 
We first show that $\PP[J_2] \geq 1 - e^{-\epsilon'_2 n}$ for some values of $\epsilon_2, C_2$ and
$\epsilon'_2$. Note that 

$$
\PP[J_2] \geq \PP[J_2 \mid N_{Q_1} \cap N_{Q_2}] \PP[N_{Q_1} \cap N_{Q_2}]
\geq \PP[J_2 \mid N_{Q_1} \cap N_{Q_2}] (1- O(e^{-\epsilon' n})).
$$

Thus, it is sufficient to show that
$\PP[J_2 \mid N_{Q_1} \cap N_{Q_2}] \geq 1 - e^{-c_1 n}$, for some constant $c_1 > 0$.

We apply Lemma \ref{lem: bipartite}. We construct the following auxiliary bipartite graph.
Consider the bipartite graph $H_1 = (A_1, A_2)$, where the partite sets $A_1$ and
$A_2$ are the set of rich vertices of $Q_1$ and $Q_2$, respectively. 
Note that conditional on $N_{Q_1} \cap N_{Q_2}$, $\min\{|A_1|, |A_2|\} \geq \frac{n}{2000d}$.
We put an edge between $v_1 \in A_1$ and $v_2 \in A_2$ in $H_1$
if at least one of the first $\lfloor \frac{d}{2} \rfloor$ private neighbors of $v_1$ under
the ordering $\tau$ is adjacent to $v_2$.

\begin{claim}
$H_1$ is a random bipartite graph where the probability
of any edge is $p_1 = 1-(1-p)^{\lfloor d/2 \rfloor}$ with the edges appearing
independently. 
\end{claim}

\begin{proof}[Proof of Claim]
We note the following:
let $v_1, v'_1$ be any elements
in $Q_1$ (not necessarily distinct) which have distinct private neighbors $w_1$ and 
$w'_1$, respectively.
Then, conditional on $N_{Q_1} \cap N_{Q_2}$, it is still the case that
$\PP[w_1v_2 \in E(G)] = \PP[w'_1v_2 \in E(G)]=p$, 
and furthermore these two events are still independent.
It follows that the probability of any edge in $H_1$ is $1-(1-p)^{\lfloor d/2 \rfloor}$
and that these edges appear independently. 
\end{proof}

It is easily seen that
$\frac{d^2}{4n} \leq p_1 \leq \frac{d^2}{n}$. We apply Lemma \ref{lem: bipartite}. 
 
Indeed, $10^{100}p_1^{-1} < 4 \cdot 10^{100} n/d^{2} < n/2000d \leq \min\{|A_1|, |A_2\}$, if $n$ is sufficiently
large. 
Thus, $H_1$ has a partial cover and $e(A_1, A_2) \in [0.99p_1|A_1||A_2|, 
1.01p_1|A_1||A_2|]$ with
probability at least $ 1- e^{-cp_1|A_1||A_2|} > 1- 
e^{-c_1 n}$, for some constant $c_1 > 0$. Let $(x_1,Y_1),..., (x_k, Y_k)$ be the set of pairs
in the partial cover.
It follows that $|Y_i| = \lceil e(A_1, A_2)/3|A_1| \rceil > d/C_2$ for some $C_2 > 0$
and at least $|A_2|/3$ of the vertices of $A_2$ are covered by the $Y_i$'s. Since
$e(A_1, A_2) < 1.01p_1|A_1||A_2|$, it follows that $k > \frac{\epsilon_2 n}{d^2}$
for some $\epsilon_2 > 0$. Finally, note that each $Y_i$ is a rich out-section. 
Indeed, each element of $Y_i$ is a vertex of $Q_2$ that is adjacent to at least
one of the private neighbors of $x_i$ and all vertices of $A_2$ are rich. 
Now, the facts $|Y_i| > d/C_2$ and 
$k > \frac{\epsilon_2 n}{d^2}$ are sufficient to establish that
$\PP[J_2 \mid N_{Q_1} \cap N_{Q_2}] \geq 1 - e^{-c_1 n}$. Thus, 
$\PP[J_2] \geq 1 - e^{-\epsilon'_2 n}$, for some $\epsilon'_2 > 0$.

% %%%%%%%
% 
%%%%%%%%%%%%%%%%%%%%%%%%%%%%%%%%%%%%%%%%%%%%%%%%%

The general case is similar. Suppose that we know that $\PP[J_{i}] \geq 1 - 
e^{-\epsilon'_{i} n}$ with the corresponding constants $C_{i}$ and $\epsilon_{i}$. 
We will prove that $\PP[J_{i+1}] \geq 1 - 
e^{-\epsilon'_{i+1} n}$, for some constant $\epsilon'_{i+1}$. Note that 
$$\PP[J_{i+1}] \geq \PP[J_{i+1} \mid J_{i} \cap N_{Q_{i+1}}]\PP[J_{i} \cap N_{Q_{i+1}}] 
\geq \PP[J_{i+1} \mid J_{i} \cap N_{Q_{i+1}}](1-e^{- \epsilon'_{i} n} - e^{-\epsilon' n}).$$

Therefore, it suffices to show that $\PP[J_{i+1} \mid J_{i} \cap N_{Q_{i+1}}] \geq 1 - e^{-c_i n}$,
for some constant $c_i > 0$.

 Suppose that $J_{i} \cap N_{Q_{i+1}}$ holds. We argue similarly as for the case $i=1$. 
 In the set $Q_i$ we will have at least 
$\epsilon_i n/d^{i}$ disjoint out-sections
each of which is rich and has size at least $d^{i-1}/C_i$ for some constants $C_i, \epsilon_i > 0$.

% By truncating, we may assume that the number of out-sections in $Q_i$
% is exactly $\lceil \epsilon_i n/d^{i} \rceil$ and each out-section has size
% exactly $\lceil d^{i-1} / C_i \rceil$. Now, contract each out-section of $Q_i$ into
% a single vertex and denote the resulting set of vertices by $Q'_i$. 

Consider the following bipartite graph $H_i =(A_i, A_{i+1})$.
For each of the disjoint, rich out-sections guaranteed by $J_i$ we will have a vertex $v$ in 
$A_i$, and $A_{i+1}$ will consist of the rich vertices of $Q_{i+1}$.
Note that conditional on $J_{i} \cap N_{Q_{i+1}}$, $\min\{|A_i|, |A_{i+1}|\} \geq 
\frac{\epsilon_i n}{d^i}$.
Let $v_i \in A_i$ and $v_{i+1} \in A_{i+1}$, and let $O_i$ be the out-section associated
with $v_i$. We put an edge between $v_i$ and $v_{i+1}$ in $H_i$ if at least one of the
first $\lceil d^{i-1}/C_i \rceil$ vertices of $O_i$ (under $\tau$) has one of its first
$\lfloor \frac{d}{2} \rfloor$ private neighbors (under $\tau$) adjacent to $v_{i+1}$.

\begin{claim}
$H_i$ is a random bipartite graph where the probability
of any edge is $p_i = 1- (1-p_1)^{\lceil d^{i-1}/C_i \rceil}$ with the edges appearing
independently. 
\end{claim}

\begin{proof}[Proof of Claim]
We note the following:
let $v_i, v'_i$ be any elements
in $Q_i$ (not necessarily distinct) which have distinct private neighbors $w_i$ and 
$w'_i$, respectively, and let $v_{i+1}$ be a vertex in $Q_{i+1}$.
Then, conditional on $J_{i} \cap N_{Q_{i+1}}$, it is still the case that
$\PP[w_iv_{i+1} \in E(G)] = \PP[w'_iv_{i+1} \in E(G)]=p$, 
and furthermore these two events are still independent.
It follows that the probability of any edge in $H_i$ is $1-(1-p_1)^{\lceil d^{i-1}/C_i \rceil}$
and that these edges appear independently. 
\end{proof}

% 
% Consider the bipartite graph $H_i = (Q'_i, Q_{i+1})$ with partite sets $Q'_i$
% and $Q_{i+1}$ where there is an edge between $v_i \in Q'_i$ and $v_{i+1} \in Q_{i+1}$
% if at least one of the $\lfloor \frac{d}{2} \rfloor$ private neighbors of at least one of the vertices
% in the out-section of $Q_i$ corresponding to $v_i$ 
% is adjacent to $v_{i+1}$. Thus, $H_i$ is a random bipartite graph where the probability
% of any edge is $p_i = 1- (1-p_1)^{\lceil d^{i-1}/C_i \rceil}$ with the edges appearing
% independently. 

It is easily seen that
$\tfrac{d^{i+1}}{4C_i n} < p_i < \tfrac{2d^{i+1}}{C_i n}$.

We again apply Lemma \ref{lem: bipartite}. 
Indeed, $10^{100}p_i^{-1} < 10^{100} \tfrac{4C_i n}{d^{i+1}} < \lceil \epsilon_i n/d^{i} \rceil = 
\min\{|A_i|, |A_{i+1}|\}$, if $n$ is sufficiently large.
Thus, $H_i$ has a partial cover and 
$e(A_i, A_{i+1}) \in [0.99p_i|A_i||A_{i+1}|, 1.01p_i|A_i||A_{i+1}|]$ with
probability at least $ 1- e^{-c p_i|A_i||A_{i+1}|} > 1- 
e^{-c_i n}$, for some constant $c_i > 0$. Let $(x_1,Y_1),..., (x_k, Y_k)$ be the set of pairs
in the partial cover.
It follows that the size of each $Y_j$ is $\lceil e(A_i, A_{i+1})/3|A_i| \rceil > d^i/C_{i+1}$
for some $C_{i+1} > 0$
and at least $|A_{i+1}|/3$ of the vertices of $A_{i+1}$ are covered by the $Y_i$'s. Since
$e(A_i, A_{i+1}) < 1.01p_i|A_i||A_{i+1}|$, it follows that $k > \frac{\epsilon_{i+1} n}{d^{i+1}}$
for some $\epsilon_{i+1} > 0$. 

Finally, note that each $Y_i$ is a rich out-section. 
Indeed, each element of $Y_i$ is a vertex of $Q_{i+1}$ that is adjacent to at least
one of the private neighbors of an out-section $O_i \subset Q_i$ associated with the vertex
$x_i$. Moreover, all the vertices of $A_{i+1}$ are rich. 
Now, the facts $|Y_i| > d^i/C_{i+1}$ and 
$k > \frac{\epsilon_{i+1} n}{d^{i+1}}$ are sufficient to establish that
$\PP[J_{i+1} \mid J_{i} \cap N_{Q_{i+1}}] \geq 1 - e^{-c_i n}$. Thus, 
$\PP[J_{i+1}] \geq 1 - e^{-\epsilon'_{i+1} n}$, for some $\epsilon'_{i+1} > 0$.

% Therefore, with probability at least $1-2e^{-\beta'_i n}$, there are at least $\frac{0.25}{0.45 p_i }$ disjoint
% out-sections in $X_{i+1}$ each of size at least $0.45p_i \epsilon n/d$. Since $p_i = \Theta (d^{i+2}/n)$,
% the number and the size of the out-sections are as desired.

% Thus, we have 
% \begin{eqnarray*}
%     \PP[J_{i+1}] &\geq& \PP[J_{i+1} \mid J_{i}](1-e^{- \epsilon'_{i} n})\\
%                  &>& (1-e^{-c_1 n}) (1-e^{- \epsilon'_{i} n})\\
%                  &>& 1 -e^{-\epsilon'_{i+1}n}
% \end{eqnarray*}
% for some constant $\epsilon'_{i+1}$, as required. 
This proves the claim (*).

Recall that $d = n^{0.1}$. Now, if $J_{10}$ holds, then we have that there exist at 
least $\epsilon_{10} n/ d^{10} = \epsilon_{10} > 0$ rich out-sections
of size at least $d^9/C_{10}$. Therefore, there is at least one rich out-section of 
size at least $\frac{n}{C_{10}d}$.
Now, since this out-section is rich, each of its vertices has at least
$\lfloor d/2 \rfloor$ private neighbors, giving at least $n/2C_{10} > n^{0.95}$ 
total private neighbors. 

Now, suppose that the event $JQ$ does not hold. Then, by the argument above,
for \emph{some} fixed sets
$J, Q_1,..., Q_{10}$, we must have $J^{c}_{10}$.
As argued at beginning of this proof, the number of ways to fix such sets is at most  
$2^{11} n^{2n^{0.91}}$ and thus, 
$\PP[JQ^{c}] \leq 2^{11} n^{2n^{0.91}} \PP[J^{c}_{10}] < 2^{11} n^{2n^{0.91}}e^{-\epsilon'_{10}n}
= o(1)$.

% Thus, it follows that
% 
% 
% \begin{eqnarray*}
% \PP[JQ^{c}] &\leq& 2^{11} n^{2n^{0.91}} \PP[R] \\ 
% &\leq& 2^{11} n^{2n^{0.91}}( \PP[\cap_{i=1}^{10} N_{Q_i} \cap R] + \PP[(\cap_{i=1}^{10} N_{Q_i})^c]) \\
% &\leq& 2^{11} n^{2n^{0.91}} \PP[R \mid \cap_{i=1}^{10} N_{Q_i}] + o(1) \\
% &\leq& 2^{11} n^{2n^{0.91}}
% \PP[J_{10}^c \mid \cap_{i=1}^{10} N_{Q_i}] + o(1)\\
% 	  &\leq& 2^{11} n^{2n^{0.91}}e^{-\epsilon'_{10}n} + o(1) \\
% 	  &=& o(1).
% \end{eqnarray*}

This completes
the proof of the lemma.

\end{proof}

\section{Concluding remarks}
Our intent in this paper was to disprove the normal graph conjecture. 
In fact, by setting $p:= n^{-1 + 1/10g}$ and identically mimicking the argument 
one can prove that for every $g$, there exist graphs of girth $g$ which are not
normal.

\section{Acknowledgements}
The authors would like to thank G\'abor Simonyi for 
his insights and sharing his knowledge around this conjecture.
Our thanks are also due to Seyyed Hosseini, Bojan Mohar and Saeed Rezaei
for the interesting discussions concerning their manuscripts. We are also
very grateful to the anonymous referee for many helpful suggestions
which improved the presentation of this paper. 
%\clearpage

\section*{References}

\end{document}